\documentclass{amsart}

\usepackage{amsmath, amsthm, amssymb, graphicx, amsfonts, hyperref}
\usepackage{color}
\usepackage{xypic, pinlabel}
\usepackage{graphicx,caption}
\usepackage{tikz}

\usepackage{enumerate}

\usepackage{ifpdf}

\usepackage[T1]{fontenc}
\usepackage{lscape}
\usepackage{ifthen}
\usepackage{amsfonts,latexsym,amssymb}
\usepackage{bm}
\usepackage{dsfont}
\usepackage{mathrsfs}
\usepackage{psfrag}

\newtheorem{thm}{Theorem}[section]
\newtheorem{lem}[thm]{Lemma}
\newtheorem{prop}[thm]{Proposition}
\newtheorem{cor}[thm]{Corollary}

\newcommand{\frag}[3][f]{
\ifthenelse{\equal{#1}{f}}
{
    \psfrag{#2}{\footnotesize#3}
}{
    \ifthenelse{\equal{#1}{s}}
    {
        \psfrag{#2}{\small#3}
    }{
        \ifthenelse{\equal{#1}{l}}
        {
            \psfrag{#2}{\large#3}
        }{
            \ifthenelse{\equal{#1}{ss}}
            {
                \psfrag{#2}{\scriptsize#3}
            }{
                \ifthenelse{\equal{#1}{n}}
                {
                    \psfrag{#2}{#3}
                }{
                }
            }
        }
    }
}
}

\definecolor{grigio}{gray}{.30}

\newcommand{\s}{\ \ \ \ }

\renewcommand{\emptyset}{\varnothing}

\newcommand{\gbra}[1]{\left\{ #1 \right\}}

\newcommand{\Z}{\mathds{Z}}

\newcommand{\Q}{\mathds{Q}}

\newcommand{\catE}{\mathscr{E}}

\newcommand{\Ml}{\textrm{ML}}

\title{Complementary legs and Rational balls}
\author{Ana G. Lecuona}
\address{Institute of Mathematics, University of Aix-Marseille, France}
\email{ana.lecuona@univ-amu.fr}
\date{}

\begin{document}

\maketitle
\begin{abstract}
In this note we study the Seifert rational homology spheres with two complementary legs, i.e.\ with a pair of invariants whose fractions add up to one. We give a complete classification of the Seifert manifolds with 3 exceptional fibers and two complementary legs which bound rational homology balls. The result translates in a statement on the sliceness of some Montesinos knots.  
\end{abstract}

\section{introduction}

Seifert manifolds are a particularly well understood set of 3--manifolds. They are oriented, closed 3--manifolds admitting a fixed point free action of $S^{1}$ and they are classified by their \emph{Seifert invariants} \cite{orlik}. They include all lens spaces. They were first introduced in \cite{seifert} and constitute one of the building blocks in the JSJ decomposition. They have been studied from the most various perspectives and are partially or completely classified from many different points of view: Stein fillability, symplectic fillability, tightness etc.  

A relevant and difficult question one can ask about a 3--manifold is whether or not it bounds a rational homology ball \cite{kirby}. In the case of Seifert manifolds there are some partial answers spread out in the literature \cite{cassonharer,GJ,pretzelYo,AG} and there is a complete classification for the lens space case \cite{lisca}. 

Seifert manifolds can be realized as the double cover of $S^{3}$ branched over Montesinos links \cite{montesinos}. If a Seifert manifold does not bound a rational homology ball, then it is a well known fact \cite[Lemma~17.2]{Ka} that the corresponding Montesinos link does not bound an embedded surface in $B^{4}$ with Euler characteristic equal to one. In the case of the link being a knot, this tells us that the knot is not slice. The famous slice--ribbon conjecture states that every slice knot bounds an immersed disk in $S^{3}$ with only ribbon singularities. A surface singularity is \emph{ribbon} if it is the identification of two segments, one contained in the interior of the surface and the other joining two different boundary points. 

Among the Seifert fibered rational homology spheres $Y$ with Seifert invariants $Y=Y(b;(\alpha_{1},\beta_{1}),\dots,(\alpha_{n},\beta_{n} ))$, $b\in\Z$, $\alpha_{i}>\beta_{i}>0$, we will be interested in this note in those having a \emph{pair of complementary legs}, that is, a pair of invariants verifying $\tfrac {\beta_{i}}{\alpha_{i}}+\tfrac {\beta_{j}}{\alpha_{j}}=1$. The case $n\leq3$ is of particular interest since for $n\leq3$ there is a one to one correspondence between the set of Seifert manifolds and the set of the corresponding Montesinos links and our results will have an interpretation also in terms of those. Our main result is the following.

\begin{thm}\label{main}
The Seifert fibered manifold $Y=Y(b;(\alpha_{1},\beta_{1}),(\alpha_{2},\beta_{2}),(\alpha_{3},\beta_{3}))$ with $\tfrac {\beta_{2}}{\alpha_{2}}+\tfrac {\beta_{3}}{\alpha_{3}}=1$ is the boundary of a rational homology ball if and only if $Y$ belongs to the list $\mathscr L$.
\end{thm}

The list $\mathscr L$ in the theorem is completely explicit and can be found in Section~\ref{thelist}. Given any set of Seifert spaces with at most 3 exceptional fibers we will simply say that a Montesinos link belongs to this set if the double cover of $S^{3}$ branched over this link belongs to it. In terms of Montesinos links we have the following result.
\begin{thm}\label{ribbon}
Each Montesinos link in the set $\mathscr L$ is the boundary of a ribbon surface $\Sigma$ such that $\chi(\Sigma)=1$. 
\end{thm}
An immediate corollary of this theorem is that the slice--ribbon conjecture is true inside the family of Montesinos knots whose double covers are Seifert fibered spaces with 3 exceptional fibers and two complementary legs.

The proof of Theorem~\ref{main} relies very heavily on the analysis previously done in \cite{lisca,MontesinosYo}. In this note we first show that a Seifert space $Y$ with complementary legs is rational homology cobordant to a lens space $L$. This allows us to study a lattice embedding problem in a very efficient way, which yields a relationship between the Seifert invariants of $Y$ and of $L$. The classification of lens spaces bounding rational homology balls in \cite{lisca} allows us to make the list $\mathscr L$ explicit.

\subsection{Organization of the paper}
In Section~\ref{s:notation} we introduce the strict minimum regarding Seifert manifolds, the language of lattice embeddings and Montesinos links. In Section~\ref{s:proofs} we carry out the proofs of the theorems in the introduction and establish the list $\mathscr L$. Finally in Seciton~\ref{s:remarks} we collect some noteworthy technical remarks about the lattice embedding problem.

\section{Notation and Conventions}\label{s:notation}

\subsection{Graphs and Seifert spaces}\label{s:graphs}
It is well known \cite{NR} that, at least with one of the two orientations, a Seifert fibered rational homology sphere defined by the invariants $Y=Y(b;(\alpha_{1},\beta_{1}),(\alpha_{2},\beta_{2}),(\alpha_{3},\beta_{3}))$ is the boundary of a 4--manifold obtained by plumbing according to the graph:
\[
  \begin{tikzpicture}[xscale=1.5,yscale=-0.5]
    \node (A0_4) at (4, 0) {$-a_{1,2}$};
    \node (A0_6) at (6, 0) {$-a_{n_{2},2}$};
    \node (A1_4) at (4, 1) {$\bullet$};
    \node (A1_5) at (5, 1) {$\dots$};
    \node (A1_6) at (6, 1) {$\bullet$};
    \node (A2_0) at (0, 2) {$-a_{n_{1},1}$};
    \node (A2_2) at (2, 2) {$-a_{1,1}$};
    \node (A2_3) at (3, 2) {$-a_{0}$};
    \node (A3_0) at (0, 3) {$\bullet$};
    \node (A3_1) at (1, 3) {$\dots$};
    \node (A3_2) at (2, 3) {$\bullet$};
    \node (A3_3) at (3, 3) {$\bullet$};
    \node (A4_4) at (4, 4) {$-a_{1,3}$};
    \node (A4_6) at (6, 4) {$-a_{n_{3},3}$};
    \node (A5_4) at (4, 5) {$\bullet$};
    \node (A5_5) at (5, 5) {$\dots$};
    \node (A5_6) at (6, 5) {$\bullet$};
    \path (A1_4) edge [-] node [auto] {$\scriptstyle{}$} (A1_5);
    \path (A3_0) edge [-] node [auto] {$\scriptstyle{}$} (A3_1);
    \path (A1_5) edge [-] node [auto] {$\scriptstyle{}$} (A1_6);
    \path (A3_3) edge [-] node [auto] {$\scriptstyle{}$} (A1_4);
    \path (A5_4) edge [-] node [auto] {$\scriptstyle{}$} (A5_5);
    \path (A3_2) edge [-] node [auto] {$\scriptstyle{}$} (A3_3);
    \path (A3_3) edge [-] node [auto] {$\scriptstyle{}$} (A5_4);
    \path (A3_1) edge [-] node [auto] {$\scriptstyle{}$} (A3_2);
    \path (A5_5) edge [-] node [auto] {$\scriptstyle{}$} (A5_6);
  \end{tikzpicture}
\]
where 
$$ \frac{\alpha_i}{\beta_i}=[a_{1,i},...,a_{n_{i},i}]:=a_{1,i}-\frac{1}{\displaystyle a_{2,i} - \frac{ \bigl. 1}{\displaystyle \ddots\ _{\displaystyle {a_{n_{i}-1,i}} -\frac {\bigl. 1}{a_{n_i,i}}}} },\ a_{j,i}\geq 2\ \mathrm{and}\ a_{0}=-b\geq 1.$$
This graph is unique and will be called the \emph{canonical} graph associated to $Y$. We will use the expression \emph{three legged canonical graph} to denote any such graph. 

If we are given a weighed graph $P$, like for example the one above, we will denote by $M_{P}$ the 4--manifold obtained by plumbing according to $P$ and by $Y_{P}$ its boundary. The incidence matrix of the graph, which represents the intersection pairing on $H_{2}(M_{P};\Z)$ with respect to the natural basis, will be denoted by $Q_{P}$. The number of vertices in $P$ coincides with $b_{2}(M_{P})$ and we will call $(\Z^{b_{2}(M_{P})},Q_{P})$ the intersection lattice associated to $P$. The vertices of $P$, which from now on will be identified with their images in $\Z^n$ and will also be called vectors, are indexed by elements of the set $J:=\{(s,\alpha)|\, s\in\{0,1,...,n_\alpha\},\ \alpha\in\{1,2,3\}\}$. Here, $\alpha$ labels the legs of the graph and $L_\alpha:=\{v_{i,\alpha}\in P|\, i=1,2,...,n_\alpha\}$ is the set of vertices of the $\alpha$-leg. 
The \emph{string} associated to the leg $L_\alpha$ is the $n_\alpha$-tuple of integers $(a_{1,\alpha},...,a_{n_\alpha,\alpha})$, where $a_{i,\alpha}:=-Q_{P}(v_{i,\alpha}, v_{i,\alpha})\geq 0$. The three legs are connected to a common central vertex, which we denote indistinctly by $v_0=v_{0,1}=v_{0,2}=v_{0,3}$ (notice that, with our notation, $v_0$ does not belong to any leg).

Let $(\Z^{n},-\mathrm{Id})$ be the standard negative diagonal lattice with the $n$ elements of a fixed basis $\catE$ labeled as $\{e_{j}\}_{j=1}^{n}$.  As an abbreviation in notation we will write $e_j\cdot e_i$ to denote $-\mathrm{Id}(e_j,e_i)$. If the intersection lattice $(\Z^{b_{2}(M_{P})},Q_{P})$ admits an embedding $\iota$ into $(\Z^n,-\mathrm{Id})$,  we will then write $P\subseteq\Z^{n}$ and will omit the $\iota$ in the notation, i.e.\ instead of writing $\iota (v)=\sum_j x_je_j$ we will directly write $v=\sum_j, x_je_j$.

We will adopt the following notation: for each $i\in\gbra{1,...,n}$, $(s,\alpha)\in J$ and  $S\subseteq P$ we define
\begin{align*}
V_S & :=\{j\in\{1,...,n\}\,|\ e_j\cdot v_{s,\alpha}\neq 0\mbox{ for some }v_{s,\alpha}\in S\}.
\end{align*}
Given $e\in\Z^n$ with $e\cdot e=-1$, we denote by $\pi_e:\Z^n\to\Z^n$ the orthogonal projector onto the subspace orthogonal to $e$, i.e.
$$\pi_e(v):=v+(v\cdot e)e\in\Z^n,\s\s\forall v\in\Z^n.$$

\subsection{Complementary legs}\label{s:cl}
The key concept in this paper is that of complementary legs. As explained in the introduction, they are pairs of legs in the graph $P$ associated to Seifert invariants with the property $\tfrac {\beta_{i}}{\alpha_{i}}+\tfrac {\beta_{j}}{\alpha_{j}}=1$. 

Suppose that in the graph displayed at the beginning of Section~\ref{s:graphs} the legs $L_{2}$ and $L_{3}$ were complementary. It is well known then that the strings of integers $(a_{1,1},...,a_{n_{1},1})$ and $(a_{1,2},...,a_{n_{2},2})$ are related to one another by Riemenschneider's point rule \cite{pointrule} and that there is an embedding of $L_{2}\cup L_{3}\subseteq\Z^{n_{2}+n_{3}}$ \cite{liscasum}. There essential features of this embedding that we will need in this note are:
\begin{itemize}
\item There exists a sequence of contractions (see \cite{lisca} for the definition and general theory of embeddings of linear sets) to the legs $\tilde L_{2}$ and $\tilde L_{3}$ with associated strings $(2)$ and $(2)$. The embedding $\tilde L_{2}\cup\tilde L_{3}\subseteq\Z^{2}$ is unique and given by $e_{1}+e_{2}$ and $e_{1}-e_{2}$.
\item The basis vector $e_{1}$ appears only twice in the embedding $L_{2}\cup L_{3}\subseteq\Z^{n_{2}+n_{3}}$.
\item Given a linear weighed graph $L$ with string $(b_{1},\dots,b_{k})$ and a embedding $L\subseteq\Z^{k}$, it is possible to combine $L$ and two complementary legs to yield a 3 legged graph $P_{L}$ with a embedding in $\Z^{k+n_{2}+n_{3}}$. Indeed, the three legs will be the complementary legs $L_{2}$ and $L_{3}$, and the leg $L_{1}$ with associated string $(b_{1},\dots,b_{k-1})$; the central vertex will have weight $b_{k}+1$ and the embedding is the obvious one, adding $-e_{1}$ to the embedding of the vertex in $L$ with weight $b_{k}$.
\item The embedding in $\Z^{k+n_{2}+n_{3}}$ of the graph $P_{L}$ constructed in the preceding point can be contracted to an embedding of $P_{L}^{c}\subseteq\Z^{k+2}$ which is a graph with $\tilde L_{2}$ and $\tilde L_3$ as complementary legs.
\end{itemize}

A final notion that will appear in the text and is related to complementary legs is that of the \emph{dual} of a string of integers $(a_{1},\dots,a_{n})$, $a_{i}\geq 2$. It refers to the unique string $(b_{1},\dots,b_{m})$, $b_{i}\geq 2$ such that 
$$
\frac{1}{[a_{1},\dots,a_{n}]}+\frac{1}{[b_{1},\dots,b_{m}]}=1.
$$

\subsection{Montesinos links}
Seifert spaces admit an involution which presets them as double covers of $S^{3}$ with branch set a link. This link can be easily recovered from any plumbing graph $P$ which provides a surgery presentation for the Seifert manifold $Y_{P}$. As shown in Figure~\ref{f:bandas} the Kirby diagram obtained from $P$ is a strongly invertible link. The involution on $Y_{P}$, which is a restriction of an involution on $M_{P}$, yields as quotient $S^{3}$, boundary of $B^{4}$, and the branch set consists of a collection of twisted bands plumbed together. The twists correspond to the weights in the graph and the plumbing to the edges. The Montesinos link is by definition the boundary of the surface obtained by plumbing bands \cite{montesinos}. An example is shown in Figure~\ref{f:bandas}.

Given a three legged graph $P$, there is a unique Montesinos link $\Ml_{P}$ which can be obtained by the above procedure. If we start with a linear graph $P'$, then the same procedure yields a 2-bridge link, which we will denote by $K(p,q)$ where $Y_{P'}$ is the lens space $L(p,q)$.
\begin{figure}[h!]
\begin{center}
\frag[n]{s}{$=$}
\frag[n]{f}{$\cong$}
\frag{(a)}{\textbf{(a)}}
\frag{(b)}{\textbf{(b)}}
\frag{1}{$1$}
\frag{2}{$2$}
\frag{p}{$-2$}
\frag{3}{$3$}
\frag{1}{$-1$}
\frag{u}{$u$}
\includegraphics[scale=0.8]{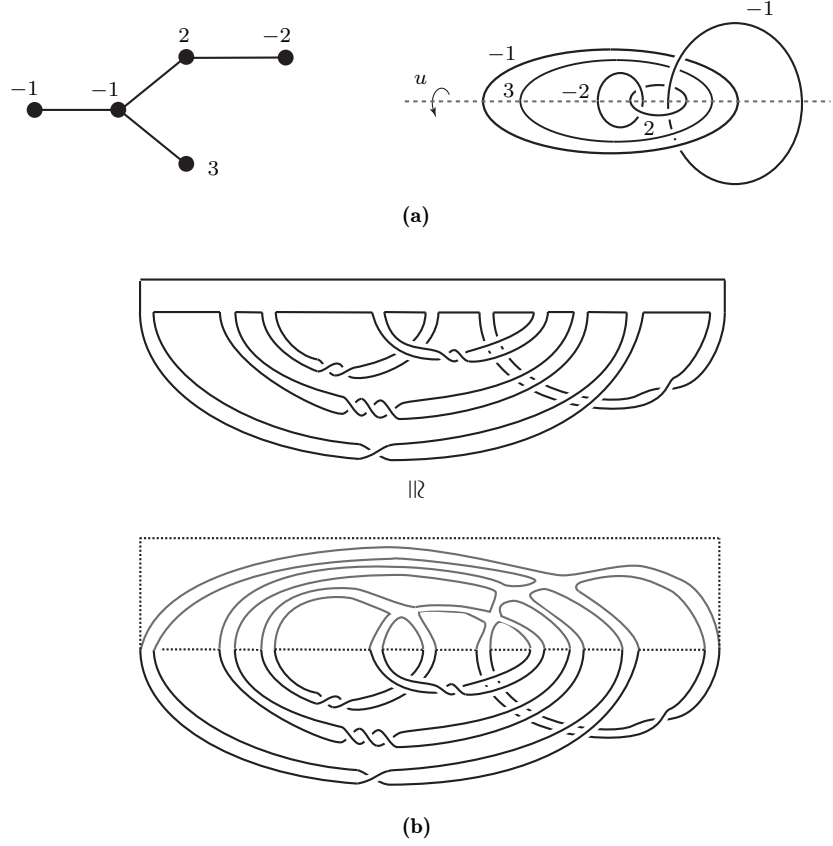}
\captionsetup{width=\linewidth}
\caption{\small Part \textbf{(a)} shows a 3 legged plumbing graph and its associated Kirby diagram as a strongly invertible link. The bottom picture shows the branch surface in $B^{4}$ which is the image of Fix($u$). This surface is homeomorphic to the result of plumbing bands according to the initial graph: the gray lines retract onto the sides of the rectangle.}
\label{f:bandas}
\end{center}
\end{figure}

\section{The list $\mathscr L$ and consequences}\label{s:proofs}

The key feature of plumbing graphs with complementary legs is the following proposition. A similar proof in a more general setting can be found in \cite[Proposition~4.6]{AcetoTesis} (cf. \cite[Remark~6.5]{pretzelYo}).

\begin{prop}\label{p:cl+ball}
Let $P$ be a 3 legged canonical graph  with two complementary legs $L_{2}$ and $L_{3}$ and such that the string associated to $L_{1}\cup\{v_{0}\}$ is of the form  $(a_{n_1,1},...,a_{1,1},a_{0})$. Then, there is a rational homology cobordism between $Y_{P}$ and the 3--manifold  $Y_{P'}$ associated to the linear graph with string $(a_{n_1,1},...,a_{1,1},a_{0}-1)$.
\end{prop}
\begin{proof}
Since the graph $P$ has two complementary legs, one can attach a 4 dimensional $2-$handle $h$ to $M_P$ obtaining a $4-$manifold whose boundary is $(S^{1}\times S^{2})\# Y_{P'}$, where $P'$ is a linear graph with associated string $(a_{n_1,1},a_{n_1-1,1},...,a_{1,1},a_{0}-1)$ \cite[Lemma~3.1]{MontesinosYo}. Notice that the dual $2-$handle $\tilde h$, attached to $(S^1\times S^2)\# Y_{P'}$, kills the free part of $H_1((S^1\times S^2)\# Y_{P'};\Z)$, since $Y_P=\partial M_P$ is a rational homology sphere. Moreover, attaching a $3-$handle $k$ to $(S^1\times S^2)\# Y_{P'}$ along one of the essentials spheres of the summand $S^1\times S^2$, the free part of $H_2((S^1\times S^2)\# Y_{P'};\Z)$ is killed and one obtains $Y_{P'}$. Dually, attaching a $1-$handle $\tilde k$ to $Y_{P'}$ yields $(S^1\times S^2)\# Y_{P'}$. 

It follows that, if $Y_{P'}$ bounds a rational homology ball $W$, then $Y_P$ is the boundary of the $4-$manifold $X:=W\cup\tilde k\cup\tilde h$, where the attaching sphere of $\tilde h$ has non-trivial algebraic intersection with the belt sphere of $\tilde k$, since $\tilde h$ kills the free part of $H_1((S^1\times S^2)\# Y_S;\Z)$. We claim that $X$ is a rational homology $4-$ball. In fact, since $W$ is a rational homology ball, we have $H_4(X;\Q)=H_3(X;\Q)=0$. The rest of the homology groups can be easily computed using cellular homology with rational coefficients: calling $X^m$ the $m-$skeleton of the CW--complex $X$, the cellular chain complex is
\begin{align*}
\xymatrix{
\cdots
\ar[r] 
&
H_3(X^3,X^2;\Q)
\ar[r]^{d_3}
&
H_2(X^2,X^1;\Q)
\ar[r]^{d_2} 
& 
H_1(X^1,X^0;\Q)
\ar[r]^(.7){d_1}
&
\Q\ .
}
\end{align*}
Since the attaching sphere of $\tilde h\in H_2(X^2,X^1;\Q)$ has non-trivial algebraic intersection with the belt sphere of $\tilde k\in H_1(X^1,X^0;\Q)$, by definition of the boundary operator $d_2$, we have $\tilde h\not\in\mathrm{Ker}\,d_2$, and since $W$ is a rational homology ball it follows that $H_2(X;\Q)=0$. Finally, we assume, without loss of generality, that $X$ is connected and hence $d_1\equiv 0$. Therefore, we have $H_1(X;\Q)=H_1(X^1,X^0;\Q)/\mathrm{im}\,d_2$. The group $H_1(X^1,X^0;\Q)$, which is a direct sum of finitely many $\Q$, has one more $\Q$ summand, coming from the $1-$handle $\tilde k$, than $H_1(W^1,W^0;\Q)$. However, $\tilde k\in\mathrm{im}\,d_2$ and therefore $H_1(X;\Q)=H_1(W;\Q)=0$. In fact, an appropriate decomposition of $X$ and $W$ gives the following commutative diagram.
\begin{align*}
\xymatrix{
H_2(W^2,W^1;\Q)
\ar[r]^{d_2}
\ar[d]^{i_\ast}
&
H_1(W^1,W^0;\Q)
\ar[r]^(.7){d_1}
\ar[d]^{i_\ast} 
& 
0
\\
H_2(X^2,X^1;\Q)
\ar[r]^{d_2}
& H_2(X^1,X^0;\Q)
}
\end{align*}
Note that the first row of the diagram is exact because $W$ is connected and $H_1(W;\Q)=0$. Since the attaching sphere of $\tilde h$ has non-trivial algebraic intersection with the belt sphere of $\tilde k$, there exists some $\alpha\in H_1(W^1,W^0;\Q)$ such that $d_2(\tilde h)=q\tilde k+i_\ast(\alpha)$, where $q\in\Q$, $q\neq 0$. The exactness of the first row implies that $\alpha=d_2(\beta)$ for some $\beta\in H_2(W^2,W^1;\Q)$. Therefore
$$d_2(\tilde h-i_\ast(\beta))=d_2(\tilde h)-d_2i_\ast (\beta)=q\tilde k+i_\ast (\alpha)-i_\ast d_2(\beta)=q\tilde k,$$
and, as claimed, $X$ is a rational homology ball.

Vice versa, if $Y_P$ bounds a rational homology ball $Z$, then $Y_{P'}$ is the boundary of a $4-$manifold $V:=Z\cup h\cup k$, where this time the attaching sphere of $k$ has non-trivial algebraic intersection with the belt sphere of $h$. It follows, arguing as before, that $V$ is a rational homology ball. We conclude that there exists a rational homology cobordism between $Y_P$ and $Y_{P'}$.
\end{proof}

If $P$ is a canonical 3 legged graph and $Y_{P}$ is the boundary of a rational homology ball $W$, then there is an embedding of the lattice associated to $H_{2}(M_{P},\Z)$ into the standard negative lattice of the same rank. Indeed, since $P$ is canonical the intersection form $Q_{P}$ of $M_{P}$ is negative definite \cite[Theorem~5.2]{NR} and hence, $X_{P}:=M_{P}\cup_{Y_{P}}(-W)$ is a closed smooth negative $4$-manifold. By Donaldson's Theorem \cite{Do} the intersection lattice of $X_{P}$ is isomorphic to $(\Z^{n},-\mathrm{Id})$, where $n=b_2(X_P)$. Moreover, since $W$ is a rational homology ball $n=b_{2}(M_{P})$ and therefore, via $X$, there is an embedding of $(H_{2}(M_{P},\Z),Q_{P})$ into $(\Z^{n},-\mathrm{Id})$, which we simply write as $P\subseteq\Z^{\vline P\vline}$. In the next lemma we study the consequences of this obstruction when the graph $P$ has two complementary legs.

\begin{lem}\label{l:gs+cl}
Let $P$ be a canonical 3 legged graph with two complementary legs $L_{2}$ and $L_{3}$ and suppose that $Y_{P}$ is the boundary of a rational homology ball. Then, the linear set $L_1\cup\{v_0\}$ has associated string of length $n_1+1$ of the form
$$(a_{n_1,1},...,a_{r,1},2^{[r]}),$$
where $0\leq r\leq n_1+1$ and the string $(a_{n_1,1},...,a_{r,1}-1)$ corresponds to a linear graph admitting an embedding in $\Z^{n_1-r+1}$ and defining a lens space $L_{P'}$.
\end{lem}
\begin{proof}
Since $Y_{P}$ bounds a rational homology ball, the space $Y_{P'}$ from Proposition~\ref{p:cl+ball} does too and therefore it holds $P'\subseteq\Z^{\vline P'\vline}$. From Section~\ref{s:cl} we know that it is always possible to extend the embedding of $P'\subseteq\Z^{\vline P'\vline}$ to an embedding $P\subseteq\Z^{\vline P\vline}$ in such a way that $V_{L_1}\cap V_{L_2\cup L_3}=\emptyset$. It follows that if $Y_{P}$ bounds a rational homology ball then there is an embedding $P\subseteq\Z^{\vline P\vline}$ with the property  $V_{L_1}\cap V_{L_2\cup L_3}=\emptyset$ which allows us to contract the complementary legs to $\widetilde L_2=\{v_{1,2}=e_j+e_k\}$ and $\widetilde L_3=\{v_{1,3}=e_j-e_k\}$ for some $j,k\in\{1,...,n\}$. 

Since $v_0\cdot v_{1,2}=v_0\cdot v_{1,3}=1$, then $j\in V_{v_0}$, $|e_j\cdot v_0|=1$ and $k\not\in V_{v_0}$. If $v\in L_1$ was such that $V_v\cap\{e_j,e_k\}\neq\emptyset$, then we could not have $v\cdot v_{1,2}=0$ and $v\cdot v_{1,3}=0$, therefore $V_{L_1}\cap V_{L_2\cup L_3}=\emptyset$. Since $v_{1,1}\cdot v_0=1$, we have $V_{v_0}\cap V_{L_1}\neq\emptyset$ and hence $v_0=\tilde v_0-e_j$ with $V_{\tilde v_0}\subseteq V_{L_1}$. Moreover, it follows that $a_0=-v_0\cdot v_0\geq 2$.

If $a_0\geq 3$ then $\tilde v_0\cdot\tilde v_0\leq -2$. By definiton of complementary legs  $|V_{L_2\cup L_3}|=n_2+n_3$ and therefore $P'_{n_1+1}:=L_1\cup\{\tilde v_0\}\subseteq\Z^{n_1+1}$. Notice that in this case the $r$ in the statement equals $0$.

If $a_0=2$ let us call $1\leq r\leq n_1+1$ the bigest integer such that $a_{0,1}=...=a_{r-1,1}=2$. It is not difficult to check that the $r$ vectors $v_{0,1},...,v_{r-1,1}$ belong to the span of $r+1$ basis vectors. If $r=n_1+1$, then $|V_{L_1\cup\{v_0\}}|=n_1+2$, by definition of complementary legs $|V_{L_2\cup L_3}|=n_2+n_3$ and $V_{L_1\cup\{v_0\}}\cap V_{L_2\cup L_3}=\{j\}$. Therefore every 3 legged canonical graph with two complementary legs $L_2$ and $L_3$ satisfying $a_{j,1}=2$ $\forall j\in\{0,...,n_1\}$ admits an embedding into $\Z^n$, where $n=n_1+n_2+n_3+1$. 

Finally, if $r<n_1+1$ there exists $i\in\{1,...,n\}$ such that $\{i\}=V_{v_{r,1}}\cap V_{v_{r-1,1}}$ and since $|V_{L_1}\cap V_{v_0}|=1$ we have that 
$$
P'_{n_1-r+1}:=(L_1\setminus\bigcup_{j=2}^{r+1} v_{j-1,1})\cup\pi_{e_i}(v_{r,1})\subseteq\Z^{n_1-r+1}.
$$
Since $v_{r-1,1}\cdot v_{r,1}=1$, we have $\pi_{e_i}(v_{r,1})\cdot\pi_{e_i}(v_{r,1})=-a_{r,1}+1$ and the statement follows.
\end{proof}

Combining Proposition~\ref{p:cl+ball} and Lemma~\ref{l:gs+cl} and using the same notation and conventions we immediately obtain:

\begin{cor}
The 3--manifold $Y_P$ bounds a rational homology ball if and only if either $r=n_1+1$ or the lens space $L_{P'}$ bounds a rational homology ball.
\end{cor}
\begin{proof}
From the proof of Lemma~\ref{l:gs+cl} we know that if $r=n_1+1$ then the string associated to $P'$ is $(2,2...,2,1)$ and therefore $Y_{P'}=S^3$ which obviously bounds a rational homology ball. On the other hand, if $r<n_1+1$ then the string is of the form $(a_{n_1,1},...,a_{r,1},2,...,2,1)$ and therefore $Y_{P'}=L_{P'}$ where $L_{P'}$ is a lens space with associated string $(a_{n_1,1},...,a_{r,1}-1)$. We conclude using Proposition~\ref{p:cl+ball}.
\end{proof}

\subsection{The List}\label{thelist}
The above analysis yields a complete list $\mathscr L$ of Seifert spaces with 3 exceptional fibers and two complementary legs bounding rational homology balls and proves Theorem~\ref{main}. Every element in $\mathscr L$ is, with one of the two orientations, the boundary of a 4--manifold obtained by plumbing according to a graph 
{\small
\[
  \begin{tikzpicture}[xscale=1.3,yscale=-0.5]
    \node (A0_6) at (6, 0) {$-a_{1,2}$};
    \node (A0_8) at (8, 0) {$-a_{n_{2},2}$};
    \node (A1_6) at (6, 1) {$\bullet$};
    \node (A1_7) at (7, 1) {$\dots$};
    \node (A1_8) at (8, 1) {$\bullet$};
    \node (A2_0) at (0, 2) {$-a_{n_{1},1}$};
    \node (A2_2) at (2, 2) {$-a_{r,1}$};
    \node (A2_3) at (3, 2) {$-2$};
    \node (A2_5) at (5, 2) {$-2$};
    \node (A3_0) at (0, 3) {$\bullet$};
    \node (A3_1) at (1, 3) {$\dots$};
    \node (A3_2) at (2, 3) {$\bullet$};
    \node (A3_3) at (3, 3) {$\bullet$};
    \node (A3_4) at (4, 3) {$\dots$};
    \node (A3_5) at (5, 3) {$\bullet$};
    \node (A4_6) at (6, 4) {$-a_{1,3}$};
    \node (A4_8) at (8, 4) {$-a_{n_{3},3}$};
    \node (A5_6) at (6, 5) {$\bullet$};
    \node (A5_7) at (7, 5) {$\dots$};
    \node (A5_8) at (8, 5) {$\bullet$};
    \path (A3_0) edge [-] node [auto] {$\scriptstyle{}$} (A3_1);
    \path (A3_3) edge [-] node [auto] {$\scriptstyle{}$} (A3_4);
    \path (A3_4) edge [-] node [auto] {$\scriptstyle{}$} (A3_5);
    \path (A1_6) edge [-] node [auto] {$\scriptstyle{}$} (A1_7);
    \path (A1_7) edge [-] node [auto] {$\scriptstyle{}$} (A1_8);
    \path (A5_6) edge [-] node [auto] {$\scriptstyle{}$} (A5_7);
    \path (A5_7) edge [-] node [auto] {$\scriptstyle{}$} (A5_8);
    \path (A3_2) edge [-] node [auto] {$\scriptstyle{}$} (A3_3);
    \path (A3_5) edge [-] node [auto] {$\scriptstyle{}$} (A5_6);
    \path (A3_5) edge [-] node [auto] {$\scriptstyle{}$} (A1_6);
    \path (A3_1) edge [-] node [auto] {$\scriptstyle{}$} (A3_2);
  \end{tikzpicture}
\]}
where the strings $(a_{1,2},...,a_{n_{2},2})$ and $(a_{1,3},...,a_{n_{3},3})$ are related to one another by Riemenschneider's point rule, the number of $-2$ in the left leg is arbitrary (including the possibility zero) and the string $(a_{r,1}-1,...,a_{n_{1},1})$, up to order reversal and duality, can be found in \cite[Remark~3.2]{MontesinosYo} which summarizes \cite[Lemmas~7.1, 7.2 and 7.3]{lisca}.


\subsection{Montesinos links and ribbon surfaces} Given that there is a one-to-one correspondence between the set of Seifert spaces with at most 3 exceptional fibers and the Montesinos links which arise as branch sets, we can translate the analysis on the rational homology balls developed so far into the language of Montesinos links and ribbon surfaces. Theorem~\ref{ribbon} in the introduction is an immediate consequence of the following proposition.

\begin{prop}\label{p:famC}
Let $P$ be a canonical graph with two complementary legs. 
The Montesinos link $\Ml_P\subset S^3$ is the boundary of a ribbon surface $\Sigma$ with $\chi(\Sigma)=1$ if and only if the Seifert space $Y_P$ is the boundary of a rational homology ball.
\end{prop}
\begin{proof}
Since $P$ has two complementary legs, we know that there is a ribbon move that changes the link $\Ml_P$ into the disjoint union $K(p,q)\sqcup U$, where $U$ stands for the unknot and $K(p,q)$ is the $2$-bridge link associated to the lens space $L_{P'}$ in the statement of Lemma~\ref{l:gs+cl} \cite[Lemma~3.3 and discussion before it]{MontesinosYo}. 

If $Y_{P}$ is the boundary of a rational homology ball, then, by Proposition~\ref{p:cl+ball}, $L_{P'}$ is also the boundary of a rational homology ball. \cite[Theorem~1.2]{lisca} tells us then that $K(p,q)\subset S^3$ is the boundary of a ribbon surface $\Sigma'$ with $\chi(\Sigma')=1$. The statement follows immediately.

Conversely, if there exist a surface $\Sigma$ and a ribbon immersion $\Sigma\looparrowright S^{3}$, with $\partial\Sigma=\Ml_P$ and $\chi (\Sigma)=1$, then it is well known that $Y_P$ is the boundary of a rational homology ball \cite{Ka}.
\end{proof}

In this last proposition we have shown that any 3 legged slice Montesinos \emph{knot} with two complementary legs is actually the boundary of a ribbon disk. We have thus:

\begin{cor}\label{c:sr+cl}
The slice-ribbon conjecture holds true for all the Montesinos knots $\Ml_P$ with $P$ a 3 legged canonical graph with 2 complementary legs.
\end{cor}

\section{noteworthy remarks on lattice embeddings}\label{s:remarks}
\begin{itemize}
\item Similar families of 3 legged graphs have been studied in \cite{MontesinosYo}. In that article all graphs $\Gamma$ such that $\Gamma\subseteq\Z^{\vline \Gamma\vline}$ define Seifert spaces which bound  rational homology balls and the corresponding Montesinos links are the boundary of ribbon surfaces with Euler characteristic $1$. In other words, Donaldson obstruction is sufficient to yield a complete classification. The situation is completely analogous in the case of linear graphs (lens spaces) \cite{lisca}. In contrast, not all the sets studied in this article, that is 3 legged graphs $P$ with two complementary legs such that $P\subseteq\Z^{\vline P\vline}$, define Seifert spaces which bound rational homology balls. For example, the Seifert space associated to 
\begin{center}
\frag{10}{$-e_5+3e_6$}
\frag{2}{$-e_1+e_3$}
\frag{3}{$-e_3+e_4+e_5$}
\frag{a}{$e_5-e_4$}
\frag{e}{$e_1+e_2$}
\frag{d}{$e_1-e_2$}
\includegraphics[scale=0.7]{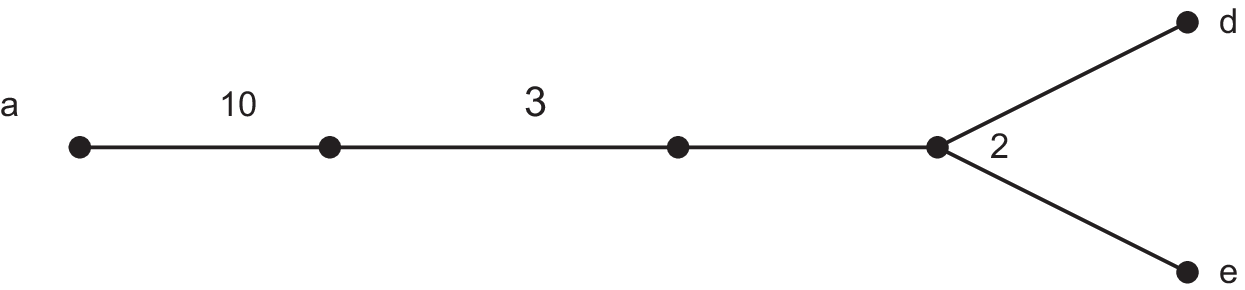}
\end{center}
is not the boundary of a rational homology ball, since the lens space $L(36,19)$ with associated string $(2,10,2)$ does not belong to \cite[Theorem~1.2]{lisca}.

\item In all the cases studied in \cite{MontesinosYo,lisca,GJ,pretzelYo} if a graph $\Gamma$ yields a 3--manifold $Y_{\Gamma}$ which bounds a rational homology ball, then all the coefficients appearing in the embedding $\Gamma\subseteq\Z^{\vline \Gamma\vline}$ can be chosen to be equal to $\pm 1$.  This is no longer true in the family studied in this note. For example,
\begin{center}
\frag{5}{$2e_4-e_3$}
\frag{2}{$e_3-e_1$}
\frag{e}{$e_1+e_2$}
\frag{d}{$e_1-e_2$}
\includegraphics[scale=0.7]{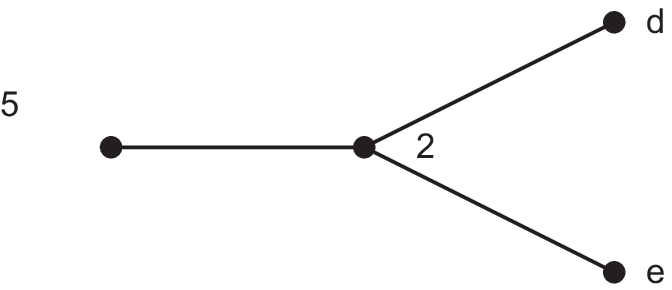}
\end{center}
satisfies $|v_{1,1}\cdot e_4|=2$ and it is the boundary of a rational homology ball, since the lens space $L(4,1)$ has this property. 

\end{itemize}

\bibliographystyle{amsplain}
\bibliography{bibliomis}

\providecommand{\MR}[1]{}
\providecommand{\bysame}{\leavevmode\hbox to3em{\hrulefill}\thinspace}
\providecommand{\MR}{\relax\ifhmode\unskip\space\fi MR }
\providecommand{\MRhref}[2]{%
  \href{http://www.ams.org/mathscinet-getitem?mr=#1}{#2}
}
\providecommand{\href}[2]{#2}
\begin{thebibliography}{10}

\bibitem{AcetoTesis}
Paolo Aceto, \emph{Rational homology cobordisms of plumbed 3 manifolds}, arXiv
  preprint: 1502.03863.

\bibitem{AG}
Paolo Aceto and Marco Golla, \emph{Dehn surgeries and rational homology balls},
  arXiv preprint: 1509.07559.

\bibitem{cassonharer}
Andrew~J. Casson and John~L. Harer, \emph{Some homology lens spaces which bound
  rational homology balls}, Pacific J. Math. \textbf{96} (1981), no.~1, 23--36.
  \MR{634760}

\bibitem{Do}
S.~K. Donaldson, \emph{The orientation of {Y}ang-{M}ills moduli spaces and
  {$4$}-manifold topology}, J. Differential Geom. \textbf{26} (1987), no.~3,
  397--428. \MR{910015}

\bibitem{Ka}
D.~Eisenbud and W.~Neumann, \emph{Three-dimensional link theory and invariants
  of plane curve singularities}, Annals of Mathematics Studies, vol. 110,
  Princeton University Press, Princeton, NJ, 1985. \MR{817982 (87g:57007)}

\bibitem{GJ}
Joshua Greene and Stanislav Jabuka, \emph{The slice-ribbon conjecture for
  3-stranded pretzel knots}, Amer. J. Math. \textbf{133} (2011), no.~3,
  555--580. \MR{2808326}

\bibitem{kirby}
Rob Kirby, \emph{Problems in low dimensional manifold theory}, Algebraic and
  geometric topology ({P}roc. {S}ympos. {P}ure {M}ath., {S}tanford {U}niv.,
  {S}tanford, {C}alif., 1976), {P}art 2, Proc. Sympos. Pure Math., XXXII, Amer.
  Math. Soc., Providence, R.I., 1978, pp.~273--312. \MR{520548}

\bibitem{MontesinosYo}
Ana~G. Lecuona, \emph{On the slice-ribbon conjecture for {M}ontesinos knots},
  Trans. Amer. Math. Soc. \textbf{364} (2012), no.~1, 233--285. \MR{2833583}

\bibitem{pretzelYo}
\bysame, \emph{On the slice-ribbon conjecture for pretzel knots}, Algebr. Geom.
  Topol. \textbf{15} (2015), no.~4, 2133--2173. \MR{3402337}

\bibitem{lisca}
Paolo Lisca, \emph{Lens spaces, rational balls and the ribbon conjecture},
  Geom. Topol. \textbf{11} (2007), 429--472. \MR{2302495}

\bibitem{liscasum}
\bysame, \emph{Sums of lens spaces bounding rational balls}, Algebr. Geom.
  Topol. \textbf{7} (2007), 2141--2164. \MR{2366190}

\bibitem{montesinos}
Jos\'e~M. Montesinos, \emph{Seifert manifolds that are ramified two-sheeted
  cyclic coverings}, Bol. Soc. Mat. Mexicana (2) \textbf{18} (1973), 1--32.
  \MR{0341467}

\bibitem{NR}
Walter~D. Neumann and Frank Raymond, \emph{Seifert manifolds, plumbing, {$\mu
  $}-invariant and orientation reversing maps}, Algebraic and geometric
  topology ({P}roc. {S}ympos., {U}niv. {C}alifornia, {S}anta {B}arbara,
  {C}alif., 1977), Lecture Notes in Math., vol. 664, Springer, Berlin, 1978,
  pp.~163--196. \MR{518415}

\bibitem{orlik}
Peter Orlik and Frank Raymond, \emph{On {$3$}-manifolds with local {${\rm
  SO}(2)$} action}, Quart. J. Math. Oxford Ser. (2) \textbf{20} (1969),
  143--160. \MR{0266214}

\bibitem{pointrule}
Oswald Riemenschneider, \emph{Deformationen von {Q}uotientensingularit\"aten
  (nach zyklischen {G}ruppen)}, Math. Ann. \textbf{209} (1974), 211--248.
  \MR{0367276}

\bibitem{seifert}
H.~Seifert, \emph{Topologie {D}reidimensionaler {G}efaserter {R}\"aume}, Acta
  Math. \textbf{60} (1933), no.~1, 147--238. \MR{1555366}

\end{thebibliography}

\end{document}